\documentclass[12pt, reqno]{amsart}
\usepackage{fullpage}

\usepackage{amsfonts,amssymb,latexsym,amsmath, amsxtra}
\usepackage{verbatim}

\newcommand{\field}[1]{\mathbb{#1}}

\newcommand{\N}{\field{N}}

\newcommand{\ov}{\overline}

\numberwithin{equation}{section}
\newtheorem{theorem}{Theorem}[section]
\newtheorem{lemma}[theorem]{Lemma}

\newtheorem{proposition}[theorem]{Proposition}
\theoremstyle{remark}

\renewenvironment{proof}[1][Proof]{\begin{trivlist}
\item[\hskip \labelsep {\bfseries #1:}]}{\qed\end{trivlist}}

\title{A generalisation of a partition theorem of Andrews to overpartitions}
\author{Jehanne Dousse}
\address{LIAFA \\
Universite Paris Diderot - Paris 7 \\
75205 Paris Cedex 13 \\
FRANCE }
\email{jehanne.dousse@liafa.univ-paris-diderot.fr}
\date{\today}

\begin{document}

\begin{abstract}
In 1969, Andrews~\cite{Generalisation1} proved a theorem on partitions with difference conditions which generalises Schur's celebrated partition identity. In this paper, we generalise Andrews' theorem to overpartitions. Our proof uses $q$-differential equations and recurrences.
\end{abstract}

\maketitle

%
%

\section{Introduction}
A partition of $n$ is a non-increasing sequence of natural numbers whose sum is $n$.
An overpartition of $n$ is a partition of $n$ in which the first occurrence of a number may be overlined.
For example, there are $14$ overpartitions of $4$:
$4$, $\overline{4}$, $3+1$, $\overline{3}+1$, $3+\overline{1}$, $\overline{3}+\overline{1}$, $2+2$, $\overline{2}+2$, $2+1+1$, $\overline{2}+1+1$, $2+\overline{1}+1$, $\overline{2}+\overline{1}+1$, $1+1+1+1$ and $\overline{1}+1+1+1$.

In 1926, Schur~\cite{Schur} proved the following partition identity.

\begin{theorem}[Schur]
\label{schur}
Let $D_1(n)$ denote the number of partitions of $n$ into distinct parts congruent to $1$ or $2$ modulo $3$.
Let $E_1(n)$ denote the number of partitions of $n$ of the form $n= \lambda_1 + \cdots + \lambda_s$ where $\lambda_i - \lambda_{i+1} \geq 3$ with strict inequality if $\lambda_{i+1} \equiv 0 \mod 3$.
Then $A(k,n)=B(k,n)$.
\end{theorem}

Several proofs of Schur's theorem have been given using a variety of different techniques such as bijective mappings~\cite{Bessenrodt,Bressoud}, the method of weighted words~\cite{Alladi}, and recurrences~\cite{Andrews2,Andrews1,Andrews3}.

Schur's theorem was subsequently generalised to overpartitions by Lovejoy~\cite{Lovejoy}, using the method of weighted words. The case $k=0$ corresponds to Schur's theorem.

\begin{theorem}[Lovejoy]
\label{schur_over}
Let $D_1(k,n)$ denote the number of overpartitions of $n$ into parts congruent to $1$ or $2$ modulo $3$ with $k$ non-overlined parts.
Let $E_1(k,n)$ denote the number of overpartitions of $n$ with $k$ non-overlined parts, where parts differ by at least $3$ if the smaller is overlined or both parts are divisible by $3$, and parts differ by at least $6$ if the smaller is overlined and both parts are divisible by $3$.
Then $D_1(k,n)=E_1(k,n)$.
\end{theorem}

Theorem~\ref{schur_over} was then proved bijectively by Raghavendra and Padmavathamma~\cite{Pad}, and using $q$-difference equations and recurrences by the author~\cite{Dousse}.

Andrews used the ideas of his proof of Schur's theorem with recurrences based on the smallest part of the partition~\cite{Andrews1} to prove a much more general theorem on partitions with difference conditions~\cite{Generalisation1}, of which another special case is the following.

\begin{theorem}[Andrews]
\label{andrews7}
Let $D_2(n)$ denote the number of partitions of $n$ into distinct parts $\equiv 1,2,4 \mod 7$. Let $E_2(n)$ denote the number of partitions of $n$ of the form $n= \lambda_1 + \cdots + \lambda_s$, where 
\begin{equation*}
\lambda_i - \lambda_{i+1} \geq 
\begin{cases}
7\ \text{if}\ \lambda_{i+1} \equiv 1,2,4\ (mod\ 7),\\
12 \ \text{if}\ \lambda_{i+1} \equiv 3\ (mod\ 7),\\
10 \ \text{if}\ \lambda_{i+1} \equiv 5,6\ (mod\ 7),\\
15 \ \text{if}\ \lambda_{i+1} \equiv 0\ (mod\ 7).
\end{cases}
\end{equation*}
Then $D_2(n) = E_2(n).$
\end{theorem}

The author took the first step towards the generalisation of Andrews' theorem to overpartitions by generalising Theorem~\ref{andrews7} in~\cite{Dousse} (again, the case $k=0$ corresponds to Theorem~\ref{andrews7}):

\begin{theorem}
\label{dousse7}
Let $D_2(k,n)$ denote the number of overpartitions of $n$ into parts $\equiv 1,2,4 \mod 7$, with $k$ non-overlined parts.
Let $E_2(k,n)$ denote the number of overpartitions of $n$ with $k$ non-overlined parts of the form $n= \lambda_1 +\cdots+ \lambda_s$, where 
\begin{equation*}
\lambda_i - \lambda_{i+1} \geq 
\begin{cases}
0 + 7 \chi(\ov{\lambda_{i+1}})\ \text{if}\ \lambda_{i+1} \equiv 1,2,4\ (mod\ 7),\\
5 + 7 \chi(\ov{\lambda_{i+1}})\ \text{if}\ \lambda_{i+1} \equiv 3\ (mod\ 7),\\
3 + 7 \chi(\ov{\lambda_{i+1}})\ \text{if}\ \lambda_{i+1} \equiv 5,6\ (mod\ 7),\\
8 + 7 \chi(\ov{\lambda_{i+1}})\ \text{if}\ \lambda_{i+1} \equiv 0\ (mod\ 7),
\end{cases}
\end{equation*}
where $\chi(\ov{\lambda_{i+1}})=1$ if $\lambda_{i+1}$ is overlined and $0$ otherwise.
Then $C(k,n)=D(k,n)$.
\end{theorem}

Now we need to introduce some notations due to Andrews in order to state his general theorem and its generalisation to overpartitions.
Let $A=\lbrace a(1), ..., a(r) \rbrace$ be a set of $r$ distinct integers such that $\sum_{i=1}^{k-1} a(i) < a(k)$ for all $1 \leq k \leq r$ and the $2^r -1$ possible sums of distinct elements of $A$ are all distinct. We denote this set of sums by $A'=\lbrace \alpha(1), ..., \alpha(2^r -1) \rbrace$, where $\alpha(1) < \cdots < \alpha(2^r-1)$. Let us notice that $\alpha(2^k)=a(k+1)$ for all $0 \leq k \leq r-1$ and that any $\alpha$ between $a(k)$ and $a(k+1)$ has largest summand $a(k)$.
Let $N$ be a positive integer with $N \geq \alpha(2^r-1) = a(1) +\cdots+a(r).$ Let $A_N$ denote the set of positive integers congruent to some $a(i) \mod N$ and $A'_N$ the set of positive integers congruent to some $\alpha(i) \mod N.$ Let $\beta_N(m)$ be the least positive residue of $m \mod N$. If $\alpha \in A'$, let $w(\alpha)$ be the number of terms appearing in the defining sum of $\alpha$ and $v(\alpha)$ the smallest $a(i)$ appearing in this sum.

To illustrate these notations in the remainder of this paper, it might be useful to consider the example where $a(k)=2^{k-1}$ for $1 \leq k \leq r$ and $\alpha(k)=k$ for $1 \leq k \leq 2^r-1$.

We are now able to state Andrews' theorem.
\begin{theorem}[Andrews]
\label{andrews}
Let $D(A_N;n)$ denote the number of partitions of $n$ into distinct parts taken from $A_N$. Let $E(A'_N;n)$ denote the number of partitions of $n$ into parts taken from $A'_N$ of the form $n=\lambda_1+\cdots+ \lambda_s$, such that
$$\lambda_i - \lambda_{i+1} \geq N w(\beta_N(\lambda_{i+1}))+v(\beta_N(\lambda_{i+1}))-\beta_N(\lambda_{i+1}).$$
Then $D(A_N;n)= E(A'_N;n)$.
\end{theorem}

As Theorems~\ref{schur} and~\ref{andrews7} generalise to overpartitions, it was interesting to know whether it is also possible to generalise Theorem~\ref{andrews}. We answer this question by proving the following.

\begin{theorem}
\label{dousse}
Let $D(A_N;k,n)$ denote the number of overpartitions of $n$ into parts taken from $A_N$, having $k$ non-overlined parts. Let $E(A'_N;k,n)$ denote the number of overpartitions of $n$ into parts taken from $A'_N$ of the form $n=\lambda_1+\cdots+ \lambda_s$, having $k$ non-overlined parts, such that
$$\lambda_i - \lambda_{i+1} \geq N w\left(\beta_N(\lambda_{i+1}) -1 +\chi(\ov{\lambda_{i+1}}) \right)+v(\beta_N(\lambda_{i+1}))-\beta_N(\lambda_{i+1}),$$
where $\chi(\ov{\lambda_{i+1}})=1$ if $\lambda_{i+1}$ is overlined and $0$ otherwise.
Then $D(A_N;k,n)= E(A'_N;k,n)$.
\end{theorem}

Theorem~\ref{schur} (resp. Theorem~\ref{schur_over}) corresponds to $N=3$, $a(1)= 1$, $a(2)=2$ and Theorem~\ref{andrews7} (resp. Theorem~\ref{dousse7}) corresponds to $N=7$, $a(1)=1$, $a(2)=2$, $a(3)=4$ in Theorem~\ref{andrews} (resp. in Theorem~\ref{dousse}).
Again, the case $k=0$ of Theorem~\ref{dousse} gives Theorem~\ref{andrews}.

The remainder of this paper is devoted to the proof of Theorem~\ref{dousse}. First, we give the $q$-differential equation satisfied by the generating function for overpartitions enumerated by $E(A'_N;k,n)$. Then we prove by induction on $r$ that a function satisfying this $q$-difference equation is equal to $ \prod_{j=1}^r \frac{(-q^{a(j)};q^N)_{\infty}}{(dq^{a(j)};q^N)_{\infty}},$ which is the generating function for overpartitions counted by $D(A_N;k,n)$. Here we use the classical notation $(a;q)_{n} = \prod_{j=0}^{n-1} (1-aq^j).$

\section{The $q$-difference equation satisfied by the generating function}

Let $p_{\alpha(i)}(k,m,n)$ denote the number of overpartitions counted by $E(A'_N;k,n)$ having $m$ parts such that the smallest part is $\geq \alpha(i)$. Let us define $\alpha(2^r):= a(r+1)= N+ a(1).$

The following lemma holds.

\begin{lemma}
\label{lemma1}
If $1 \leq i \leq 2^r-1$, then
\begin{equation}
\label{eq1}
\begin{aligned}
p_{\alpha(i)}(k,m,n) &- p_{\alpha(i+1)}(k,m,n)
\\ =& ~p_{v(\alpha(i))}\big(k,m-1,n-(m-1)Nw(\alpha(i)) - \alpha(i)\big)
\\+& p_{v(\alpha(i))}\big(k-1,m-1,n-(m-1)N(w(\alpha(i))-1) - \alpha(i)\big),
\end{aligned}
\end{equation}

\begin{equation}
\label{eq2}
p_{\alpha(2^r)} (k,m,n) = p_{a(1)}(k,m,n-mN).
\end{equation}
\end{lemma}

\begin{proof}
Let us start by proving~\eqref{eq1}.
We observe that $p_{\alpha(i)}(k,m,n) - p_{\alpha(i+1)}(k,m,n)$ is the number of overpartitions of the form $n= \lambda_1 +\cdots+ \lambda_m$ enumerated by $p_{\alpha(i)}(k,m,n)$ such that the smallest part is equal to $\alpha(i)$.

If $\lambda_m = \ov{\alpha(i)}$ is overlined, then by definition of $E(A'_N;k,n)$,
$$\lambda_{m-1} \geq \alpha(i)+ N w(\alpha(i)) + v(\alpha(i)) - \alpha(i) =  N w(\alpha(i)) + v(\alpha(i)).$$
In that case we remove $\lambda_m=\ov{\alpha(i)}$ and subtract $N w(\alpha(i))$ from each remaining part. The number of parts is reduced to $m-1$, the number of non-overlined parts is still $k$, and the number partitioned is now $n-(m-1)N w(\alpha(i))-\alpha(i)$. Moreover the smallest part is now $\geq v(\alpha(i)).$ Therefore we have an overpartition counted by $p_{v(\alpha(i))}(k,m-1,n-(m-1)Nw(\alpha(i)) - \alpha(i))$.

If $\lambda_m = \alpha(i)$ is not overlined, then by definition of $E(A'_N;k,n)$,
$$\lambda_{m-1} \geq N \left(w(\alpha(i))-1\right) + v(\alpha(i)).$$
In that case we remove $\lambda_m=\alpha(i)$ and subtract $N (w(\alpha(i))-1)$ from each remaining part. The number of parts is reduced to $m-1$, the number of non-overlined parts is reduced $k-1$, and the number partitioned is now $n-(m-1)N (w(\alpha(i))-1)-\alpha(i)$. Moreover the smallest part is now $\geq v(\alpha(i)).$ Therefore we have an overpartition counted by $p_{v(\alpha(i))}(k-1,m-1,n-(m-1)N(w(\alpha(i))-1) - \alpha(i))$.

To prove~\eqref{eq2}, we consider a partition enumerated by $p_{\alpha(2^r)} (k,m,n)$ and subtract $N$ from each part. As $p_{\alpha(2^r)} (k,m,n) = p_{N+a(1)}(k,m,n)$, we obtain a partition enumerated by $p_{a(1)} (k,m,n-N).$
\end{proof}

For $|d|<1$, $|x|<1$, $|q|<1$, we define
\begin{equation}
\label{def_fi}
f_{\alpha(i)}(d,x,q)=f_{\alpha(i)}(x):=1+ \sum_{n=1}^{\infty} \sum_{m=1}^{\infty} \sum_{k=0}^{\infty} p_{\alpha(i)} (k,m,n) d^k x^m q^n.
\end{equation}

We want to find $f_{a(1)}(1)$, which is the generating function for all overpartitions counted by $E(A'_N;k,n)$. To do so, we  establish a $q$-difference equation relating $f_{a(1)}\left(xq^{jN}\right)$, for $j \geq 0$. Let us start by giving some relations between generating functions.

Lemma~\ref{lemma1} directly implies

\begin{lemma}
\label{lemma2}
If $1 \leq i \leq 2^r-1$, then
\begin{equation}
\label{eqf1}
f_{\alpha(i)}(x) - f_{\alpha(i+1)}(x)= xq^{\alpha(i)} f_{v(\alpha(i))}\left(xq^{Nw(\alpha(i))}\right) + dxq^{\alpha(i)} f_{v(\alpha(i))}\left(xq^{N\big(w(\alpha(i))-1 \big)}\right),
\end{equation}

\begin{equation}
\label{eqf2}
f_{\alpha(2^r)}(x) = f_{a(1)}\left(xq^N\right).
\end{equation}
\end{lemma}
\noindent Adding equations~\eqref{eqf1} together for $1 \leq i \leq 2^{k-1}-1$ and using the fact that $\alpha\left(2^{k-1}\right)=a(k)$, we obtain

\begin{equation}
\label{eq3.5}
f_{a(1)}(x) - f_{a(k)}(x) =\sum_{\alpha < a(k)} \left( xq^{\alpha} f_{v(\alpha)}\left(xq^{Nw(\alpha)}\right) + dxq^{\alpha} f_{v(\alpha)}\left(xq^{N(w(\alpha)-1)}\right)\right).
\end{equation}
Let us now add equations~\eqref{eqf1} together for $2^{k-2} \leq i \leq 2^{k-1}-1$. This gives

\begin{equation}
\label{eq3.6}
f_{a(k-1)}(x) - f_{a(k)}(x) =\sum_{a(k-1) \leq \alpha < a(k)} \left( xq^{\alpha} f_{v(\alpha)}\left(xq^{Nw(\alpha)}\right) + dxq^{\alpha} f_{v(\alpha)}\left(xq^{N(w(\alpha)-1)}\right)\right).
\end{equation}
Every $a(k-1) < \alpha < a(k)$ is of the form $\alpha = a(k-1) + \alpha',$ with $\alpha' < a(k-1).$
Hence we can rewrite~\eqref{eq3.6} as

\begin{align*}
f_{a(k-1)}(x) &- f_{a(k)}(x)
\\&= xq^{a(k-1)} f_{a(k-1)}\left(xq^{N}\right) + dxq^{a(k-1)} f_{a(k-1)}\left(x\right)
\\&+ q^{a(k-1)-N} \sum_{\alpha' < a(k-1)} \left( xq^{\alpha'+N} f_{v(\alpha')}\left(xq^{N(w(\alpha')+1)}\right)+ dxq^{\alpha'+N} f_{v(\alpha')}\left(xq^{Nw(\alpha')}\right)\right)
\\&= xq^{a(k-1)} f_{a(k-1)}\left(xq^{N}\right) + dxq^{a(k-1)} f_{a(k-1)}\left(x\right)
\\&+ q^{a(k-1)-N} \left( f_{a(1)}\left(xq^N\right) - f_{a(k-1)}\left(xq^N\right) \right),
\end{align*}
where the last equality follows from~\eqref{eq3.5}.

Thus
\begin{equation}
\label{eq3.7}
\begin{aligned}
f_{a(k)}(x) &= \left(1-dxq^{a(k-1)}\right) f_{a(k-1)}(x) - q^{a(k-1)-N} f_{a(1)}\left(xq^N\right)
\\&+ q^{a(k-1)-N}\left(1-xq^N\right) f_{a(k-1)}\left(xq^N\right).
\end{aligned}
\end{equation}

Remember we want to establish a $q$-difference equation relating functions $f_{a(1)}\left(xq^{kN}\right)$ for $k \geq 0$. Before this, we must recall some facts about $q$-binomial coefficients defined by
$${m \brack r}_q :=
\begin{cases}
\frac{\left(1-q^m\right)\left(1-q^{m-1}\right) \dots \left(1-q^{m-r+1}\right)}{\left(1-q\right) \left(1-q^2\right) \dots \left(1-q^r\right)}\ \text{if}\ 0 \leq r \leq m,\\
0 \ \text{otherwise}.
\end{cases}$$
They are $q$-analogues to binomial coefficients and satisfy $q$-analogues of the Pascal triangle identity~\cite{Gasper}.

\begin{proposition}
\label{pascal}
For all integers $0 \leq r \leq m$,
\begin{equation}
\label{pascal1}
{m \brack r}_q = q^r {m-1 \brack r}_q + {m-1 \brack r-1}_q,
\end{equation}
\begin{equation}
\label{pascal2}
{m \brack r}_q ={m-1 \brack r}_q + q^{m-r} {m-1 \brack r-1}_q.
\end{equation}
\end{proposition}
As $q \rightarrow 1$ this is exactly Pascal's identity.

We are now ready to state the key lemma which will lead to the desired $q$-difference equation.
\begin{lemma}
\label{conj}
For $1 \leq k \leq r+1$, we have
\begin{equation}
\label{eq}
\begin{aligned}
\prod_{j=1}^{k-1} &\left(1-dxq^{a(j)}\right) f_{a(1)}(x) = f_{a(k)}(x) 
\\ + \sum_{j=1}^{k-1} &\left( \sum_{m=0}^{k-j-1} d^m \sum_{\substack{\alpha < a(k) \\ w(\alpha)=j+m}} xq^{\alpha} \left( (-x)^{m-1} {j+m-1 \brack m-1}_{q^N} + (-x)^m {j+m \brack m}_{q^N} \right) \right)
\\ &\times \prod_{h=1}^{j-1} \left(1-xq^{hN}\right) f_{a(1)}\left(xq^{jN}\right).
\end{aligned}
\end{equation}
\end{lemma}
\begin{proof}
We prove this lemma by induction on $k$.
For $k=1$, this reduces to $f_{a(1)}(x) = f_{a(1)}(x).$
Let us assume that~\eqref{eq} is true for some $1 \leq k \leq r$ and show it also holds for $k+1$.
In the following let
\begin{align*}
s_k(x) : = \sum_{j=1}^{k-1} &\left( \sum_{m=0}^{k-j-1} d^m \sum_{\substack{\alpha < a(k) \\ w(\alpha)=j+m}} xq^{\alpha} \left( (-x)^{m-1} {j+m-1 \brack m-1}_{q^N} + (-x)^m {j+m \brack m}_{q^N} \right) \right)
\\ &\times \prod_{h=1}^{j-1} \left(1-xq^{hN}\right) f_{a(1)}\left(xq^{jN}\right).
\end{align*}
Therefore we want to prove that
$$\prod_{j=1}^{k} \left(1-dxq^{a(j)}\right) f_{a(1)}(x) = f_{a(k+1)}(x) + s_{k+1}(x).$$
We have
\begin{align*}
\prod_{j=1}^{k} &\left(1-dxq^{a(j)}\right) f_{a(1)}(x) - f_{a(k+1)}(x)
\\=& \left(1-dxq^{a(k)}\right) \left( \prod_{j=1}^{k-1} \left(1-dxq^{a(j)}\right) f_{a(1)}(x) - f_{a(k)}(x) \right)
\\&+ \left(1-dxq^{a(k)}\right) f_{a(k)}(x) -f_{a(k+1)}(x)
\\=& \left(1-dxq^{a(k)}\right) s_k(x)
\\&+ q^{a(k)-N} f_{a(1)}\left(xq^N\right) - q^{a(k)-N}\left(1-xq^N \right) f_{a(k)}(xq^N),
\end{align*}
where the last equality follows from the induction hypothesis and equation~\ref{eq3.7}.
Thus
\begin{align*}
\prod_{j=1}^{k} &\left(1-dxq^{a(j)}\right) f_{a(1)}(x) - f_{a(k+1)}(x)
\\=& \left(1-dxq^{a(k)}\right) s_k(x)+ q^{a(k)-N} f_{a(1)}\left(xq^N\right)
\\&- q^{a(k)-N}\left(1-xq^N \right) \left( \prod_{j=1}^{k-1} \left(1-dxq^{N+a(j)}\right) f_{a(1)}\left(xq^N\right) -s_k\left(xq^N\right)\right)
\\=& \left(1-dxq^{a(k)}\right) s_k(x)+ q^{a(k)-N}\left(1-xq^N \right) s_k\left(xq^N\right)
\\&+ q^{a(k)-N}  \left( 1- \left(1-xq^N \right) \prod_{j=1}^{k-1} \left(1-dxq^{N+a(j)}\right)\right) f_{a(1)}\left(xq^N\right)
\\=& \left(1-dxq^{a(k)}\right) s_k(x)+ q^{a(k)-N}\left(1-xq^N \right) s_k\left(xq^N\right)
\\&+ q^{a(k)-N}  \left( 1- \left(1-xq^N \right) \left( 1 + \sum_{m=1}^{k-1}\sum_{\substack{\alpha < a(k) \\ w(\alpha)=m}} (-dxq^N)^mq^{\alpha}\right) \right) f_{a(1)}\left(xq^N\right)
\\=& \left(1-dxq^{a(k)}\right) s_k(x)+ q^{a(k)-N}\left(1-xq^N \right) s_k\left(xq^N\right)
\\&+ q^{a(k)-N}  \left( xq^N  + \sum_{m=1}^{k-1} d^m \sum_{\substack{\alpha < a(k) \\ w(\alpha)=m}}  xq^{\alpha+N} \left((-xq^N)^{m-1} + (-xq^N)^m\right) \right) f_{a(1)}\left(xq^N\right)
\\=& \left(1-dxq^{a(k)}\right) s_k(x)+ q^{a(k)-N}\left(1-xq^N \right) s_k\left(xq^N\right)
\\&+  \left( xq^{a(k)}  + \sum_{m=1}^{k-1} d^m \sum_{\substack{a(k) <\alpha' < a(k+1) \\ w(\alpha')=m+1}}  xq^{\alpha'} \left((-xq^N)^{m-1} + (-xq^N)^m\right) \right) f_{a(1)}\left(xq^N\right)
\end{align*}
\begin{align*}
=& \sum_{j=1}^{k-1} \left( \sum_{m=0}^{k-j-1} d^m \sum_{\substack{\alpha < a(k) \\ w(\alpha)=j+m}} xq^{\alpha} \left( (-x)^{m-1} {j+m-1 \brack m-1}_{q^N} + (-x)^m {j+m \brack m}_{q^N} \right) \right)
\\ &\times \prod_{h=1}^{j-1} \left(1-xq^{hN}\right) f_{a(1)}\left(xq^{jN}\right)
\\+& \sum_{j=1}^{k-1} \left( \sum_{m=0}^{k-j-1} d^{m+1} \sum_{\substack{\alpha < a(k) \\ w(\alpha)=j+m}} xq^{a(k)+\alpha} \left( (-x)^{m} {j+m-1 \brack m-1}_{q^N} + (-x)^{m+1} {j+m \brack m}_{q^N} \right) \right)
\\ &\times \prod_{h=1}^{j-1} \left(1-xq^{hN}\right) f_{a(1)}\left(xq^{jN}\right)
\\+& q^{a(k)-N}\left(1-xq^N \right) \times
\\&\sum_{j=1}^{k-1 }\left( \sum_{m=0}^{k-j-1} d^m \sum_{\substack{\alpha < a(k) \\ w(\alpha)=j+m}} xq^{\alpha+N} \left( (-xq^N)^{m-1} {j+m-1 \brack m-1}_{q^N} + (-xq^N)^m {j+m \brack m}_{q^N} \right) \right)
\\&\times \prod_{h=1}^{j-1} \left(1-xq^{(h+1)N}\right) f_{a(1)}\left(xq^{(j+1)N}\right)
\\+& \left( xq^{a(k)}  + \sum_{m=1}^{k-1} d^m \sum_{\substack{a(k) <\alpha' < a(k+1) \\ w(\alpha')=m+1}}  xq^{\alpha'} \left((-xq^N)^{m-1} + (-xq^N)^m\right) \right) f_{a(1)}\left(xq^N\right)
\end{align*}
\begin{align*}
=& \sum_{j=1}^{k-1} \left( \sum_{m=0}^{k-j-1} d^m \sum_{\substack{\alpha < a(k) \\ w(\alpha)=j+m}} xq^{\alpha} \left( (-x)^{m-1} {j+m-1 \brack m-1}_{q^N} + (-x)^m {j+m \brack m}_{q^N} \right) \right)
\\ &\times \prod_{h=1}^{j-1} \left(1-xq^{hN}\right) f_{a(1)}\left(xq^{jN}\right)
\\+& \sum_{j=1}^{k-1} \left( \sum_{m=1}^{k-j} d^{m} \sum_{\substack{a(k)<\alpha < a(k+1) \\ w(\alpha)=j+m}} xq^{\alpha} \left( (-x)^{m-1} {j+m-2 \brack m-2}_{q^N} + (-x)^{m} {j+m-1 \brack m-1}_{q^N} \right) \right)
\\ &\times \prod_{h=1}^{j-1} \left(1-xq^{hN}\right) f_{a(1)}\left(xq^{jN}\right)
\\+&\sum_{j=2}^{k}\left( \sum_{m=0}^{k-j} d^m \sum_{\substack{a(k)<\alpha < a(k+1) \\ w(\alpha)=j+m}} xq^{\alpha} \left( (-xq^N)^{m-1} {j+m-2 \brack m-1}_{q^N} + (-xq^N)^m {j+m-1 \brack m}_{q^N} \right) \right)
\\&\times \prod_{h=1}^{j-1} \left(1-xq^{hN}\right) f_{a(1)}\left(xq^{jN}\right)
\\+& \left( xq^{a(k)}  + \sum_{m=1}^{k-1} d^m \sum_{\substack{a(k) <\alpha' < a(k+1) \\ w(\alpha')=m+1}}  xq^{\alpha'} \left((-xq^N)^{m-1} + (-xq^N)^m\right) \right) f_{a(1)}\left(xq^N\right)
\end{align*}
\begin{align*}
=& \sum_{j=1}^{k-1} \left( \sum_{m=0}^{k-j-1} d^m \sum_{\substack{\alpha < a(k) \\ w(\alpha)=j+m}} xq^{\alpha} \left( (-x)^{m-1} {j+m-1 \brack m-1}_{q^N} + (-x)^m {j+m \brack m}_{q^N} \right) \right)
\\ &\times \prod_{h=1}^{j-1} \left(1-xq^{hN}\right) f_{a(1)}\left(xq^{jN}\right)
\\+& \sum_{j=1}^{k-1} \left( \sum_{m=1}^{k-j} d^{m} \sum_{\substack{a(k)<\alpha < a(k+1) \\ w(\alpha)=j+m}} xq^{\alpha} \left( (-x)^{m-1} {j+m-2 \brack m-2}_{q^N} + (-x)^{m} {j+m-1 \brack m-1}_{q^N} \right) \right)
\\ &\times \prod_{h=1}^{j-1} \left(1-xq^{hN}\right) f_{a(1)}\left(xq^{jN}\right)
\\+&\sum_{j=1}^{k}\left( \sum_{m=0}^{k-j} d^m \sum_{\substack{a(k) \leq \alpha < a(k+1) \\ w(\alpha)=j+m}} xq^{\alpha} \left( (-xq^N)^{m-1} {j+m-2 \brack m-1}_{q^N} + (-xq^N)^m {j+m-1 \brack m}_{q^N} \right) \right)
\\&\times \prod_{h=1}^{j-1} \left(1-xq^{hN}\right) f_{a(1)}\left(xq^{jN}\right)
\end{align*}
\begin{align*}
= \sum_{j=1}^{k-1} \left[ \sum_{\substack{\alpha < a(k+1) \\ w(\alpha)=j}}\right. & xq^{\alpha}
\\+ \sum_{m=1}^{k-j} d^m  &\left( \sum_{\substack{\alpha < a(k) \\ w(\alpha)=j+m}} xq^{\alpha} \left( (-x)^{m-1} {j+m-1 \brack m-1}_{q^N} + (-x)^m {j+m \brack m}_{q^N} \right) \right.
\\&+ \sum_{\substack{a(k)<\alpha < a(k+1) \\ w(\alpha)=j+m}} xq^{\alpha} \left[ (-x)^{m-1} \left({j+m-2 \brack m-2}_{q^N}+ q^{N(m-1)}{j+m-2 \brack m-1}_{q^N} \right) \right.
\\&\; \; \; \; \; \; \; \; \; \; \; \; \; \; \; \; \; \; \; \; \; \; \; \; \; \; \; \; \; \; + \left.\left.\left. (-x)^{m} \left({j+m-1 \brack m-1}_{q^N} + q^{Nm} {j+m-1 \brack m}_{q^N} \right) \right] \vphantom{ \sum_{\substack{\alpha < a(k+1) \\ w(\alpha)=j}}} \right) \right]
\\\times \prod_{h=1}^{j-1} \big(1-&xq^{hN}\big) f_{a(1)}\left(xq^{jN}\right)
\\+ xq^{a(1)+ \cdots +a(k)} &\prod_{h=1}^{k-1} \left(1-xq^{hN}\right) f_{a(1)}\left(xq^{kN}\right).
\end{align*}
Thus by~\eqref{pascal1} of Proposition~\ref{pascal}, we obtain
\begin{align*}
\prod_{j=1}^{k} &\left(1-dxq^{a(j)}\right) f_{a(1)}(x) - f_{a(k+1)}(x)
\\=&\sum_{j=1}^{k-1} \left( \sum_{m=0}^{k-j} d^m   \sum_{\substack{\alpha < a(k+1) \\ w(\alpha)=j+m}} xq^{\alpha} \left( (-x)^{m-1} {j+m-1 \brack m-1}_{q^N} + (-x)^m {j+m \brack m}_{q^N} \right) \right)
\\ &\times \prod_{h=1}^{j-1} \big(1-xq^{hN}\big) f_{a(1)}\left(xq^{jN}\right)
\\+& xq^{a(1)+\cdots+a(k)} \prod_{h=1}^{k-1} \left(1-xq^{hN}\right) f_{a(1)}\left(xq^{kN}\right)
\\=&\sum_{j=1}^{k} \left( \sum_{m=0}^{k-j} d^m   \sum_{\substack{\alpha < a(k+1) \\ w(\alpha)=j+m}} xq^{\alpha} \left( (-x)^{m-1} {j+m-1 \brack m-1}_{q^N} + (-x)^m {j+m \brack m}_{q^N} \right) \right)
\\ &\times \prod_{h=1}^{j-1} \big(1-xq^{hN}\big) f_{a(1)}\left(xq^{jN}\right)
\\=&s_{k+1}(x).
\end{align*}
This completes the proof.
\end{proof}

Now, by setting $k=r+1$ in Lemma~\ref{conj}, and using~\eqref{eqf2}, we obtain the desired $q$-difference equation.
\begin{equation}
\label{qdiff}
\tag{$\mathrm{eq}_{N,r}$}
\begin{aligned}
\prod_{j=1}^{r} &\left(1-dxq^{a(j)}\right) f_{a(1)}(x) = f_{a(1)}(xq^N) 
\\ + \sum_{j=1}^{r} &\left( \sum_{m=0}^{r-j} d^m \sum_{\substack{\alpha < a(r+1) \\ w(\alpha)=j+m}} xq^{\alpha} \left( (-x)^{m-1} {j+m-1 \brack m-1}_{q^N} + (-x)^m {j+m \brack m}_{q^N} \right) \right)
\\ &\times \prod_{h=1}^{j-1} \left(1-xq^{hN}\right) f_{a(1)}\left(xq^{jN}\right).
\end{aligned}
\end{equation}

We now need to evaluate $f_{a(1)}(1)$, which we recall is the generating function for the overpartitions with difference conditions counted by $E(A'_N;k,n)$.

\section{Evaluating $f_{a(1)}(1)$ by induction}

In this section, we evaluate $f_{a(1)}(1)$. To do so, we prove by induction on $r$ the following theorem. A similar idea was already used in the proof of Theorem~\ref{dousse7} in~\cite{Dousse}, which used the proof of Theorem~\ref{dousse}. But the technical details of the following proof are much more intricate.

\begin{theorem}
\label{main}
Let $r$ be a positive integer. Then for every $N \geq \alpha(2^r-1)$, for every function $f$ satisfying $(\mathrm{eq}_{N,r})$ and the initial condition $f(0)=1$, we have
$$f(1)= \prod_{k=1}^r \frac{(-q^{a(k)};q^N)_{\infty}}{(dq^{a(k)};q^N)_{\infty}}.$$
\end{theorem}

The idea of the proof is to start from a function satisfying $(\mathrm{eq}_{N,r})$ and to do some transformations to relate it to a function satisfying $(\mathrm{eq}_{N,r-1})$ in order to use the induction hypothesis.
In order to simplify the proof, we split it into several lemmas.

\begin{lemma}
\label{lemmaF}
Let $f$ and $F$ be two functions such that
$$F(x):= f(x) \prod_{n=0}^{\infty} \frac{1-dxq^{Nn+a(r)}}{1-xq^{Nn}}.$$
Then $f(0)=1$ and $f$ satisfies~\eqref{qdiff} if and only if $F(0)=1$ and $F$ satisfies the following $q$-difference equation
\begin{equation}
\label{qdiffF}
\tag{$\mathrm{eq}'_{N,r}$}
\begin{aligned}
&\left(1 + \sum_{j=1}^{r} \left( d^{j-1} \sum_{\substack{ \alpha < a(r) \\ w(\alpha)=j-1}} q^{\alpha} +d^j \sum_{\substack{ \alpha < a(r) \\ w(\alpha)=j}} q^{\alpha} \right) (-x)^j \right) F(x)
\\&= F\left(xq^N\right) + \sum_{j=1}^r \sum_{l=1}^r \sum_{k=0}^{\min(j-1,l-1)}c_{k,j}b_{l-k,j} (-1)^{l-1}x^l F\left(xq^{jN}\right),
\end{aligned}
\end{equation}
where
$$c_{k,j}:= q^{N \frac{k(k+1)}{2} + k a(r)} {j-1 \brack k}_{q^N} d^k,$$
and
$$b_{m,j}:= \left( d^{m-1} \sum_{\substack{ \alpha < a(r+1) \\ w(\alpha)=j+m-1}} q^{\alpha} +d^m \sum_{\substack{ \alpha < a(r+1) \\ w(\alpha)=j+m}} q^{\alpha} \right) {j+m-1 \brack m-1}_{q^N}.$$
\end{lemma}
\begin{proof}
Directly plugging the definition of $f$ into $(\mathrm{eq}_{N,r})$, we get
\begin{align*}
(1-x) \prod_{j=1}^{r-1} &\left(1-dxq^{a(j)}\right) F(x) = F(xq^N) 
\\ + \sum_{j=1}^{r} &\left( \sum_{m=0}^{r-j} d^m \sum_{\substack{\alpha < a(r+1) \\ w(\alpha)=j+m}} xq^{\alpha} \left( (-x)^{m-1} {j+m-1 \brack m-1}_{q^N} + (-x)^m {j+m \brack m}_{q^N} \right) \right)
\\ &\times \prod_{h=1}^{j-1} \left(1-dxq^{hN+a(r)}\right) F\left(xq^{jN}\right).
\end{align*}
With the conventions that
$$ \sum_{\substack{\alpha < a(r) \\ w(\alpha)=n}} q^{\alpha} = 0 \ \text{for}\ n \geq r,$$
and
$$ \sum_{\substack{\alpha < a(r) \\ w(\alpha)=0}} q^{\alpha} = 1,$$
this can be reformulated as
\begin{align*}
&\left(1 + \sum_{j=1}^{r} \left( d^{j-1} \sum_{\substack{ \alpha < a(r) \\ w(\alpha)=j-1}} q^{\alpha} +d^j \sum_{\substack{ \alpha < a(r) \\ w(\alpha)=j}} q^{\alpha} \right) (-x)^j \right) F(x) = F(xq^N) 
\\ &+ \sum_{j=1}^{r} \left( \sum_{m=1}^{r-j+1} \left(d^{m-1} \sum_{\substack{\alpha < a(r+1) \\ w(\alpha)=j+m-1}} q^{\alpha} + d^m \sum_{\substack{\alpha < a(r+1) \\ w(\alpha)=j+m}} q^{\alpha}\right) {j+m-1 \brack m-1}_{q^N} (-1)^{m-1} x^m \right)
\\ &\times \left(\sum_{k=0}^{j-1} q^{N \frac{k(k-1)}{2}+k a(r)} {j-1 \brack k}_{q^N} d^k (-x)^k\right) F\left(xq^{jN}\right),
\end{align*}
because of the $q$-binomial theorem~\cite{Gasper}
\begin{equation}
\label{identity}
\prod_{k=0}^{n-1}(1+q^kt) = \sum_{k=0}^n q^{\frac{k(k-1)}{2}} {n \brack k}_q t^k,
\end{equation}
in which we replace $q$ by $q^N$, $n$ by $j-1$ and $t$ by $-dxq^{N+a(r)}$.
Finally, noting that $b_{l-k,j} = 0$ if $j+l-k-1 \geq r$, we can rewrite this as~\eqref{qdiffF}.
Moreover, $F(0)=f(0)=1$ and the lemma is proved.
\end{proof}

We can directly transform ~\eqref{qdiffF} into a recurrence equation on the coefficients of the generating function $F$.
\begin{lemma}
\label{lemmaA}
Let $F$ be a function and $(A_n)_{n \in \N}$ a sequence such that $$F(x) =: \sum_{n=0}^{\infty} A_n x^n.$$
Then $F$ satisfies $(\mathrm{eq}'_{N,r})$ and the initial condition $F(0)=1$ if and only if $A_0=1$ and $(A_n)_{n \in \N}$ satisfies the following recurrence equation
\begin{equation}
\label{recA}
\tag{$\mathrm{rec}_{N,r}$}
\begin{aligned}
& \left(1-q^{nN}\right) A_n =
\\ & \sum_{m=1}^{r} \left( d^{m-1} \sum_{\substack{\alpha < a(r) \\ w(\alpha)=m-1}} q^{\alpha} + d^{m} \sum_{\substack{\alpha < a(r) \\ w(\alpha)=m}} q^{\alpha} + \sum_{j=1}^r \sum_{k=0}^{\min(j-1,m-1)} c_{k,j} b_{m-k,j}q^{jN(n-m)} \right) (-1)^{m+1} A_{n-m}.
\end{aligned}
\end{equation}
\end{lemma}
\begin{proof}
By the definition of $(A_n)_{n \in \N}$ and~\eqref{qdiffF}, we have
\begin{align*}
\left(1-q^{nN}\right) A_n =& \sum_{j=1}^{r} \left( d^{j-1} \sum_{\substack{\alpha < a(r) \\ w(\alpha)=j-1}} q^{\alpha} + d^{j} \sum_{\substack{\alpha < a(r) \\ w(\alpha)=j}} q^{\alpha} \right) (-1)^{j+1} A_{n-j}
\\ &+  \sum_{j=1}^r \sum_{l=1}^r \sum_{k=0}^{\min(j-1,l-1)} c_{k,j} b_{l-k,j}q^{jN(n-l)} (-1)^{l+1} A_{n-l}.
\end{align*}
Relabelling the summation indices and factorising leads to~\eqref{recA}. Moreover, $A_n = F(0)=1$. This completes the proof.
\end{proof}

Let us now do some transformations starting from $(\mathrm{eq}_{N,r-1})$.

\begin{lemma}
\label{lemmaG}
Let $g$ and $G$ be two functions such that
$$G(x):= g(x) \prod_{n=0}^{\infty} \frac{1}{1-xq^{Nn}}.$$
Then $g$ satisfies $(\mathrm{eq}_{N,r-1})$ and the initial condition $g(0)=1$ if and only if $G(0)=1$ and $G$ satisfies the following $q$-difference equation
\begin{equation}
\label{qdiffG}
\tag{$\mathrm{eq}''_{N,r-1}$}
\begin{aligned}
&\left(1 + \sum_{j=1}^{r} \left( d^{j-1} \sum_{\substack{ \alpha < a(r) \\ w(\alpha)=j-1}} q^{\alpha} +d^j \sum_{\substack{ \alpha < a(r) \\ w(\alpha)=j}} q^{\alpha} \right) (-x)^j \right) G(x) = G\left(xq^N\right) 
\\&+ \sum_{j=1}^r \sum_{m=1}^{r-j} \left(d^{m-1} \sum_{\substack{ \alpha < a(r) \\ w(\alpha)=j+m-1}} q^{\alpha} +d^m \sum_{\substack{ \alpha < a(r) \\ w(\alpha)=j+m}} q^{\alpha} \right) {j+m-1 \brack m-1}_{q^N} (-1)^{m+1} x^m  G\left(xq^{jN}\right).
\end{aligned}
\end{equation}
\end{lemma}
\begin{proof}
Using the definition of $G$ and $(\mathrm{eq}_{N,r-1})$, we get
\begin{align*}
&(1-x) \prod_{j=1}^{r-1} \left(1-dxq^{a(j)}\right) G(x) = G(xq^N) 
\\&+ \sum_{j=1}^{r-1} \left( \sum_{m=0}^{r-j-1} d^m \sum_{\substack{\alpha < a(r) \\ w(\alpha)=j+m}} xq^{\alpha} \left( (-x)^{m-1} {j+m-1 \brack m-1}_{q^N} + (-x)^m {j+m \brack m}_{q^N} \right) \right) G\left(xq^{jN}\right).
\end{align*}
Then, as in the proof of Lemma~\ref{lemmaF}, this can be reformulated as~\eqref{qdiffG}, and $G(0)= g(0)=1.$
\end{proof}

Again, let us translate this into a recurrence equation on the coefficients of the generating function $G$.

\begin{lemma}
\label{lemmaa}
Let $G$ be a function and $(a_n)_{n \in \N}$ be a sequence such that
$$G(x) =: \sum_{n=0}^{\infty} a_n x^n.$$
Then $G$ satisfies $(\mathrm{eq}''_{N,r-1})$ and the initial condition $G(0)=1$ if and only if $a_0=1$ and $(a_n)_{n \in \N}$ satisfies the following recurrence equation
\begin{equation}
\label{reca}
\tag{$\mathrm{rec''}_{N,r-1}$}
\begin{aligned}
& \left(1-q^{nN}\right) a_n =
\\ & \sum_{m=1}^{r} \sum_{j=0}^{r-1} \left(d^{m-1} \sum_{\substack{\alpha < a(r) \\ w(\alpha)=j+m-1}} q^{\alpha} + d^{m} \sum_{\substack{\alpha < a(r) \\ w(\alpha)=j+m}} q^{\alpha} \right) {j+m-1 \brack m-1}_{q^N} q^{jN(n-m)} (-1)^{m+1} a_{n-m}.
\end{aligned}
\end{equation}
\end{lemma}
\begin{proof}
By the definition of $(a_n)_{n \in \N}$ and~\eqref{qdiffG}, we have
\begin{align*}
& \left(1-q^{nN}\right) a_n = \sum_{m=1}^{r} \left( d^{m-1} \sum_{\substack{\alpha < a(r) \\ w(\alpha)=m-1}} q^{\alpha} + d^{m} \sum_{\substack{\alpha < a(r) \\ w(\alpha)=m}} q^{\alpha} \right) (-1)^{m+1} a_{n-m}
\\ &+ \sum_{m=1}^{r-1} \sum_{j=1}^{r-1}  \left(d^{m-1} \sum_{\substack{\alpha < a(r) \\ w(\alpha)=j+m-1}} q^{\alpha} + d^{m} \sum_{\substack{\alpha < a(r) \\ w(\alpha)=j+m}} q^{\alpha} \right) {j+m-1 \brack m-1}_{q^N} q^{jN(n-m)} (-1)^{m+1} a_{n-m}.
\end{align*}
As the summand of the second term equals $0$ when $m=r$, we can equivalently write that the second sum is taken on $m$ going from $1$ to $r$. Then we observe that the first term corresponds to $j=0$ in the second term, and factorising gives exactly~\eqref{recA}. Moreover, $a_n = G(0)=1$. This completes the proof.
\end{proof}

Let us do a final transformation and obtain yet another recurrence equation.

\begin{lemma}
\label{lemmaA'}
Let $(a_n)_{n \in \N}$ and $(A'_n)_{n \in \N}$ be two sequences such that
$$ A'_n := a_n \prod_{k=0}^{n-1} \left( 1 +q^{Nk+a(r)} \right).$$
Then $(a_n)_{n \in \N}$ satisfies $(\mathrm{rec}''_{N,r-1})$ and the initial condition $a_0=1$ if and only if $A'_0=1$ and $(A'_n)_{n \in \N}$ satisfies the following recurrence equation
\begin{equation}
\label{recA'}
\tag{$\mathrm{rec'}_{N,r-1}$}
\begin{aligned}
\left(1-q^{nN}\right) A'_n =& \sum_{m=1}^{r} \left( \sum_{\nu=0}^{r-1} \sum_{\mu=0}^{\min(m-1, \nu)} f_{m,\mu} e_{m,\nu - \mu} q^{\nu N (n-m)} \right.
\\&+ \left. q^{a(r)} \sum_{\nu=1}^{r} \sum_{\mu=0}^{\min(m-1, \nu-1)} f_{m,\mu} e_{m,\nu - \mu -1} q^{\nu N (n-m)}\right) (-1)^{m+1}  A'_{n-m},
\end{aligned}
\end{equation}
where
$$e_{m,j} := \left(d^{m-1} \sum_{\substack{\alpha < a(r) \\ w(\alpha)=j+m-1}} q^{\alpha} + d^{m} \sum_{\substack{\alpha < a(r) \\ w(\alpha)=j+m}} q^{\alpha} \right) {j+m-1 \brack m-1}_{q^N},$$
and
$$f_{m,k} := q^{N \frac{k(k+1)}{2} +ka(r)} {m-1 \brack k}_{q^N}.$$
\end{lemma}
\begin{proof}
By definition of $(A'_n)_{n \in \N}$, we have
\begin{align*}
\left(1-q^{nN}\right) A'_n =& \sum_{m=1}^{r} \sum_{j=0}^{r-1} \left(d^{m-1} \sum_{\substack{\alpha < a(r) \\ w(\alpha)=j+m-1}} q^{\alpha} + d^{m} \sum_{\substack{\alpha < a(r) \\ w(\alpha)=j+m}} q^{\alpha} \right) {j+m-1 \brack m-1}_{q^N} q^{jN(n-m)}
\\& \times (-1)^{m+1} \prod_{k=1}^m \left(1 + q^{N(n-k)+a(r)}\right) A'_{n-m}.
\end{align*}
Furthermore
\begin{align*}
\prod_{k=1}^m \left(1 + q^{N(n-k)+a(r)}\right) &= \prod_{k=0}^{m-1} \left(1 + q^{Nk+ N(n-m)+a(r)}\right)
\\&= \left( 1+q^{N(n-m)	+a(r)}\right) \prod_{k=1}^{m-1} \left(1 + q^{Nk+ N(n-m)+a(r)}\right)
\\&= \left( 1+q^{N(n-m)	+a(r)}\right) \sum_{k=0}^{m-1} q^{N \frac{k(k+1)}{2} +kN(n-m) +ka(r)} {m-1 \brack k}_{q^N},
\end{align*}
where the last equality follows from~\eqref{identity}.
Therefore
\begin{align*}
& \left(1-q^{nN}\right) A'_n =
\\ & \sum_{m=1}^{r} \left(\sum_{j=0}^{r-1} \left(d^{m-1} \sum_{\substack{\alpha < a(r) \\ w(\alpha)=j+m-1}} q^{\alpha} + d^{m} \sum_{\substack{\alpha < a(r) \\ w(\alpha)=j+m}} q^{\alpha} \right) {j+m-1 \brack m-1}_{q^N} q^{jN(n-m)} \right.
\\& \left. \vphantom{\sum_{\substack{\alpha < a(r) \\ w(\alpha)=j+m-1}}}  \times  \left( 1+q^{N(n-m)	+a(r)}\right) \sum_{k=0}^{m-1} q^{N \frac{k(k+1)}{2} +kN(n-m) +ka(r)} {m-1 \brack k}_{q^N} \right) (-1)^{m+1} A'_{n-m}
\end{align*}
\begin{align*}
= \sum_{m=1}^{r} &\left[ \sum_{j=0}^{r-1} \left(d^{m-1} \sum_{\substack{\alpha < a(r) \\ w(\alpha)=j+m-1}} q^{\alpha} + d^{m} \sum_{\substack{\alpha < a(r) \\ w(\alpha)=j+m}} q^{\alpha} \right) {j+m-1 \brack m-1}_{q^N} q^{jN(n-m)} \right.
\\& \; \; \; \; \; \; \; \; \; \; \times \sum_{k=0}^{m-1} q^{N \frac{k(k+1)}{2} +kN(n-m) +ka(r)} {m-1 \brack k}_{q^N}
\\&+ \sum_{j=0}^{r-1} \left(d^{m-1} \sum_{\substack{\alpha < a(r) \\ w(\alpha)=j+m-1}} q^{\alpha} + d^{m} \sum_{\substack{\alpha < a(r) \\ w(\alpha)=j+m}} q^{\alpha} \right) q^{a(r)} {j+m-1 \brack m-1}_{q^N} q^{(j+1)N(n-m)}
\\& \; \; \; \; \; \; \; \; \; \;  \times \left. \vphantom{\sum_{\substack{\alpha < a(r) \\ w(\alpha)=j+m-1}}} \sum_{k=0}^{m-1} q^{N \frac{k(k+1)}{2} +kN(n-m) +ka(r)} {m-1 \brack k}_{q^N} \right] (-1)^{m+1} A'_{n-m}.
\end{align*}
Thus
\begin{align*}
\left(1-q^{nN}\right) A'_n =\sum_{m=1}^{r} &\left(\sum_{j=0}^{r-1} e_{m,j} q^{jN(n-m)} \sum_{k=0}^{m-1} f_{m,k} q^{kN(n-m)} \right.
\\&+ \left. q^{a(r)} \sum_{j=1}^{r} e_{m,j-1} q^{jN(n-m)} \sum_{k=0}^{m-1} f_{m,k} q^{kN(n-m)} \right) (-1)^{m+1} A'_{n-m}.
\end{align*}
Rearranging leads to~\eqref{recA'}. As always, $A'_0 = a_0=1$. The lemma is proved.
\end{proof}

We now want to show that $(A_n)_n \in \N$ and $(A'_n)_n \in \N$ are in fact equal.

\begin{lemma}
\label{equalAA'}
Let $(A_n)_{n \in \N}$ and $(A'_n)_{n \in \N}$ be defined as in Lemmas~\ref{lemmaA} and~\ref{lemmaA'}.
Then for every $n \in \N$, $A_n = A'_n$.
\end{lemma}
\begin{proof}
To prove the equality, it is sufficient to show that for every $1 \leq m \leq r,$ the coefficient of $(-1)^{m+1} A_{n-m}$ in~\eqref{recA} is the same as the coefficient of $(-1)^{m+1} A'_{n-m}$ in~\eqref{recA'}.
Let $m \in \lbrace 1,...,r \rbrace$ and
\begin{align*}
S_{m} &: = \left[ (-1)^{m+1} A_{n-m}\right] (\mathrm{rec}_{N,r}) 
\\&= d^{m-1} \sum_{\substack{\alpha < a(r) \\ w(\alpha)=m-1}} q^{\alpha} + d^{m} \sum_{\substack{\alpha < a(r) \\ w(\alpha)=m}} q^{\alpha} + \sum_{j=1}^r \sum_{k=0}^{\min(j-1,m-1)} c_{k,j} b_{m-k,j} q^{jN(n-m)}
\end{align*}
and
\begin{align*}
S'_{m} &: = \left[ (-1)^{m+1} A'_{n-m}\right] (\mathrm{rec'}_{N,r-1}) 
\\&= \sum_{\nu=0}^{r-1} \sum_{\mu=0}^{\min(m-1, \nu)} f_{m,\mu} e_{m,\nu - \mu} q^{\nu N (n-m)} + q^{a(r)} \sum_{\nu=1}^{r} \sum_{\mu=0}^{\min(m-1, \nu-1)} f_{m,\mu} e_{m,\nu - \mu -1} q^{\nu N (n-m)}
\\&= f_{m,0} e_{m,0} + \sum_{\nu=1}^{r} \left( \sum_{\mu=0}^{\min(m-1, \nu)} f_{m,\mu} e_{m,\nu - \mu} + q^{a(r)} \sum_{\mu=0}^{\min(m-1, \nu-1)} f_{m,\mu} e_{m,\nu - \mu -1} \right) q^{\nu N (n-m)},
\end{align*}
because $e_{m, r-\mu}=0$ for all $\mu$, as $\mu \leq m-1$ so the sums are over $\alpha$ such that $\alpha < a(r)$ and $w(\alpha) \geq r$, which is impossible.

Let us first notice that $$f_{m,0} e_{m,0} = d^{m-1} \sum_{\substack{\alpha < a(r) \\ w(\alpha)=m-1}} q^{\alpha} + d^{m} \sum_{\substack{\alpha < a(r) \\ w(\alpha)=m}} q^{\alpha}.$$
Now define
$$T_{m,j}:= \sum_{k=0}^{\min(j-1,m-1)} c_{k,j} b_{m-k,j},$$
and
$$T'_{m,j}:= \sum_{k=0}^{\min(m-1, j)} f_{m,k} e_{m,j-k} + q^{a(r)} \sum_{k=0}^{\min(m-1,j-1)} f_{m,k} e_{m,j-k-1}.$$
The only thing left to do is to show that for every $1 \leq j \leq r$, $$T_{m,j}= T'_{m,j}.$$
We have
\begin{equation}
\label{eqcb}
\begin{aligned}
&c_{k,j} b_{m-k,j} 
\\=& q^{N \frac{k(k+1)}{2}+k a(r)} {j-1 \brack k}_{q^N}  \left( d^{m-1} \sum_{\substack{ \alpha < a(r+1) \\ w(\alpha)=j+m-k-1}} q^{\alpha} +d^m \sum_{\substack{ \alpha < a(r+1) \\ w(\alpha)=j+m-k}} q^{\alpha} \right) {j+m-k-1 \brack m-k-1}_{q^N}
\\=& \left( d^{m-1} \sum_{\substack{ \alpha < a(r) \\ w(\alpha)=j+m-k-1}} q^{\alpha} +d^m \sum_{\substack{ \alpha < a(r) \\ w(\alpha)=j+m-k}} q^{\alpha} \right) q^{N \frac{k(k+1)}{2}+k a(r)} {j-1 \brack k}_{q^N}  {j+m-k-1 \brack m-k-1}_{q^N}
\\&+q^{a(r)} \left( d^{m-1} \sum_{\substack{ \alpha < a(r) \\ w(\alpha)=j+m-k-2}} q^{\alpha} +d^m \sum_{\substack{ \alpha < a(r) \\ w(\alpha)=j+m-k-1}} q^{\alpha} \right) q^{N \frac{k(k+1)}{2}+k a(r)} 
\\&\times {j-1 \brack k}_{q^N}  {j+m-k-1 \brack m-k-1}_{q^N},
\end{aligned}
\end{equation}
in which the last equality follows from separating the sums over $\alpha$ according to whether $\alpha$ contains $a(r)$ as a summand or not.

We also have
\begin{equation}
\label{eqfe}
\begin{aligned}
f_{m,k} e_{m,j-k} &= q^{N \frac{k(k+1)}{2}+k a(r)} {m-1 \brack k}_{q^N} 
\\&\times \left( d^{m-1} \sum_{\substack{ \alpha < a(r) \\ w(\alpha)=j+m-k-1}} q^{\alpha} +d^m \sum_{\substack{ \alpha < a(r) \\ w(\alpha)=j+m-k}} q^{\alpha} \right) {j+m-k-1 \brack m-1}_{q^N},
\end{aligned}
\end{equation}
and
\begin{equation}
\label{eqfe'}
\begin{aligned}
q^{a(r)} f_{m,k} e_{m,j-k-1} &= q^{N \frac{k(k+1)}{2}+(k+1) a(r)} {m-1 \brack k}_{q^N}  
\\& \times \left( d^{m-1} \sum_{\substack{ \alpha < a(r) \\ w(\alpha)=j+m-k-2}} q^{\alpha} +d^m \sum_{\substack{ \alpha < a(r) \\ w(\alpha)=j+m-k-1}} q^{\alpha} \right) {j+m-k-2 \brack m-1}_{q^N}.
\end{aligned}
\end{equation}

By a simple calculation using the definition of $q$-binomial coefficients, we get the following result
For all $j,k,m \in \N$,
\begin{equation}
\label{equalityqbin}
{m-1 \brack k}_{q^N} {j+m-k-1 \brack m-1}_{q^N} = {j \brack k}_{q^N} {j+m-k-1 \brack m-k-1}_{q^N}.
\end{equation}
Using~\eqref{equalityqbin}, we obtain
\begin{align*}
T'_{m,j}&= \chi( j \leq m-1) \ q^{N \frac{j(j+1)}{2}+j a(r)} \left( d^{m-1} \sum_{\substack{ \alpha < a(r) \\ w(\alpha)=m-1}} q^{\alpha} +d^m \sum_{\substack{ \alpha < a(r) \\ w(\alpha)=m}} q^{\alpha} \right) {m-1 \brack m-j-1}_{q^N}
\\&+ \sum_{k=0}^{\min(m-1,j-1)} q^{N \frac{k(k+1)}{2}+k a(r)}
\\& \; \; \; \; \;\; \; \; \; \; \; \; \; \; \;\times \left( d^{m-1} \sum_{\substack{ \alpha < a(r) \\ w(\alpha)=j+m-k-1}} q^{\alpha} +d^m \sum_{\substack{ \alpha < a(r) \\ w(\alpha)=j+m-k}} q^{\alpha} \right) {j \brack k}_{q^N} {j+m-k-1 \brack m-k-1}_{q^N}
\\&+ \sum_{k=0}^{\min(m-1,j-1)} q^{N \frac{k(k+1)}{2}+(k+1) a(r)}
\\& \; \; \; \; \;\; \; \; \; \; \; \; \; \; \;\times \left( d^{m-1} \sum_{\substack{ \alpha < a(r) \\ w(\alpha)=j+m-k-2}} q^{\alpha} +d^m \sum_{\substack{ \alpha < a(r) \\ w(\alpha)=j+m-k-1}} q^{\alpha} \right) {j-1 \brack k}_{q^N} {j+m-k-2 \brack m-k-1}_{q^N}.
\end{align*}
By~\eqref{pascal2} of Lemma~\ref{pascal}, we have
$${j \brack k}_{q^N} = {j-1 \brack k}_{q^N} + q^{N(j-k)} {j-1 \brack k-1}_{q^N},$$
$${j+m-k-2 \brack m-k-1}_{q^N} = {j+m-k-1 \brack m-k-1}_{q^N} - q^{Nj} {j+m-k-2 \brack m-k-2}_{q^N}.$$
This allows us to rewrite $T'_{m,j}$ as
\begin{align*}
T'_{m,j}&= \chi( j \leq m-1) \ q^{N \frac{j(j+1)}{2}+j a(r)} \left( d^{m-1} \sum_{\substack{ \alpha < a(r) \\ w(\alpha)=m-1}} q^{\alpha} +d^m \sum_{\substack{ \alpha < a(r) \\ w(\alpha)=m}} q^{\alpha} \right) {m-1 \brack m-j-1}_{q^N}
\\&+ \sum_{k=0}^{\min(m-1,j-1)} q^{N \frac{k(k+1)}{2}+k a(r)}
\\& \; \; \; \; \;\; \; \; \; \; \; \; \; \; \;\times \left( d^{m-1} \sum_{\substack{ \alpha < a(r) \\ w(\alpha)=j+m-k-1}} q^{\alpha} +d^m \sum_{\substack{ \alpha < a(r) \\ w(\alpha)=j+m-k}} q^{\alpha} \right) {j-1 \brack k}_{q^N} {j+m-k-1 \brack m-k-1}_{q^N}
\\&+ \sum_{k=0}^{\min(m-1,j-1)} q^{N \frac{k(k+1)}{2}+k a(r)+ N(j-k)}
\\& \; \; \; \; \;\; \; \; \; \; \; \; \; \; \;\times \left( d^{m-1} \sum_{\substack{ \alpha < a(r) \\ w(\alpha)=j+m-k-1}} q^{\alpha} +d^m \sum_{\substack{ \alpha < a(r) \\ w(\alpha)=j+m-k}} q^{\alpha} \right) {j-1 \brack k-1}_{q^N} {j+m-k-1 \brack m-k-1}_{q^N}
\\&+ \sum_{k=0}^{\min(m-1,j-1)} q^{N \frac{k(k+1)}{2}+(k+1) a(r)}
\\& \; \; \; \; \;\; \; \; \; \; \; \; \; \; \;\times \left( d^{m-1} \sum_{\substack{ \alpha < a(r) \\ w(\alpha)=j+m-k-2}} q^{\alpha} +d^m \sum_{\substack{ \alpha < a(r) \\ w(\alpha)=j+m-k-1}} q^{\alpha} \right) {j-1 \brack k}_{q^N} {j+m-k-1 \brack m-k-1}_{q^N}
\\&- \sum_{k=0}^{\min(m-2,j-1)} q^{N \frac{k(k+1)}{2}+(k+1) a(r)+Nj}
\\& \; \; \; \; \;\; \; \; \; \; \; \; \; \; \;\times \left( d^{m-1} \sum_{\substack{ \alpha < a(r) \\ w(\alpha)=j+m-k-2}} q^{\alpha} +d^m \sum_{\substack{ \alpha < a(r) \\ w(\alpha)=j+m-k-1}} q^{\alpha} \right) {j-1 \brack k}_{q^N} {j+m-k-2 \brack m-k-2}_{q^N}.
\end{align*}
By~\eqref{eqcb}, the sum of the second and fourth term in the sum above is exactly equal to $T{m,j}$.
Let $X$ denote the sum of the third and fifth term. We now want to show that $$X + \chi( j \leq m-1) \ q^{N \frac{j(j+1)}{2}+j a(r)} \left( d^{m-1} \sum_{\substack{ \alpha < a(r) \\ w(\alpha)=m-1}} q^{\alpha} +d^m \sum_{\substack{ \alpha < a(r) \\ w(\alpha)=m}} q^{\alpha} \right) {m-1 \brack m-j-1}_{q^N}=0.$$
By the change of variable $k'=k+1$ in the fourth sum, we get
\begin{align*}
X &= \sum_{k=0}^{\min(m-1,j-1)} q^{N \frac{k(k-1)}{2}+k a(r)+ Nj}
\\& \; \; \; \; \;\; \; \; \; \; \; \; \; \; \;\times \left( d^{m-1} \sum_{\substack{ \alpha < a(r) \\ w(\alpha)=j+m-k-1}} q^{\alpha} +d^m \sum_{\substack{ \alpha < a(r) \\ w(\alpha)=j+m-k}} q^{\alpha} \right) {j-1 \brack k-1}_{q^N} {j+m-k-1 \brack m-k-1}_{q^N}
\\&- \sum_{k=1}^{\min(m-1,j)} q^{N \frac{k(k-1)}{2}+k a(r)+Nj}
\\& \; \; \; \; \;\; \; \; \; \; \; \; \; \; \;\times \left( d^{m-1} \sum_{\substack{ \alpha < a(r) \\ w(\alpha)=j+m-k-1}} q^{\alpha} +d^m \sum_{\substack{ \alpha < a(r) \\ w(\alpha)=j+m-k}} q^{\alpha} \right) {j-1 \brack k-1}_{q^N} {j+m-k-1 \brack m-k-1}_{q^N}
\\&= \begin{cases}
0,\  \text{if}\ j \geq m,\\
-q^{N \frac{j(j+1)}{2}+j a(r)} \left( d^{m-1} \sum_{\substack{ \alpha < a(r) \\ w(\alpha)=m-1}} q^{\alpha} +d^m \sum_{\substack{ \alpha < a(r) \\ w(\alpha)=m}} q^{\alpha} \right) {m-1 \brack m-j-1}_{q^N},\ \text{otherwise}
\end{cases}
\\&= -\chi( j \leq m-1) \ q^{N \frac{j(j+1)}{2}+j a(r)} \left( d^{m-1} \sum_{\substack{ \alpha < a(r) \\ w(\alpha)=m-1}} q^{\alpha} +d^m \sum_{\substack{ \alpha < a(r) \\ w(\alpha)=m}} q^{\alpha} \right) {m-1 \brack m-j-1}_{q^N}.
\end{align*}
This completes the proof.
\end{proof}

We can finally turn to the proof of Theorem~\ref{main}.

\begin{proof}[Proof of Theorem~\ref{main}]
Let us start by the initial case $r=1$. Let $N \geq a(1)$ and $f$ such that
\begin{equation}
\tag{$\mathrm{eq}_{N,1}$}
\left(1-dxq^{a(1)}\right)f(x) = f \left(xq^N\right) +xq^{a(1)}f\left(xq^N\right).
\end{equation}

Then
\begin{equation}
\label{r1}
f(x) = \frac{1+xq^{a(1)}}{1-dxq^{a(1)}} f\left(xq^N\right).
\end{equation}
Iterating~\eqref{r1}, we get
$$f(x)= \prod_{n=0}^{\infty} \frac{1+xq^{Nn+a(1)}}{1-dxq^{Nn+a(1)}} f(0).$$
Thus
$$f(1)=\frac{(-q^{a(1)};q^N)_{\infty}}{(dq^{a(1)};q^N)_{\infty}}.$$

Now assume that Theorem~\ref{main} is true for some $r-1 \geq 1$. We want to show that it is true for $r$ too.
Let $N \geq \alpha \left( 2^r -1 \right)$ and $f$ be a function with $f(0)=1$ satisfying~\eqref{qdiff}.
Let
$$F(x):= f(x) \prod_{n=0}^{\infty} \frac{1-dxq^{Nn+a(r)}}{1-xq^{Nn}}.$$
By Lemma~\ref{lemmaF}, $F(0)=1$ and $F$ satisfies~\eqref{qdiffF}.
Now let $$F(x) =: \sum_{n=0}^{\infty} A_n x^n.$$
Then by Lemma~\ref{lemmaA} $A_0=1$ and $(A_n)_{n \in \N}$ satisfies~\eqref{recA}.
But by Lemma~\ref{equalAA'}, $(A_n)_{n \in \N}$ also satisfies~\eqref{recA'}.
Now let  $$ A_n =: a_n \prod_{k=0}^{n-1} \left( 1 +q^{Nk+a(r)} \right).$$
By Lemma~\ref{lemmaA'}, $a_0 =1$ and $(a_n)_{n \in \N}$ satisfies~\eqref{reca}.
Let $$G(x) := \sum_{n=0}^{\infty} a_n x^n.$$
By Lemma~\ref{lemmaa}, $G(0)=1$ and $G$ satisfies~\eqref{qdiffG}.
Finally let $$g(x)=: G(x) \prod_{n=0}^{\infty} \left(1-xq^{Nn}\right).$$
By Lemma~\ref{lemmaG}, $g(0)=1$ and $g$ satisfies $(\mathrm{eq}_{N,r-1})$.
Now $N$ is still larger than $\alpha \left( 2^{r-1}-1 \right)$ and we can use the induction hypothesis which gives
\begin{equation}
\label{g1}
g(1)= \prod_{k=1}^{r-1} \frac{(-q^{a(k)};q^N)_{\infty}}{(dq^{a(k)};q^N)_{\infty}}.
\end{equation}
By Appell's comparison theorem~\cite{Appell},
\begin{align*}
\lim_{x \rightarrow 1^-} (1-x) \sum_{n=0}^{\infty} a_n x^n &= \lim_{n \rightarrow \infty} a_n
\\&= \lim_{x \rightarrow 1^-} (1-x) G(x) 
\\&= \lim_{x \rightarrow 1^-} (1-x) \frac{g(x)}{\prod_{n=0}^{\infty} \left(1-xq^{Nn}\right)}
\\&= \frac{g(1)}{\prod_{n=1}^{\infty}\left(1-q^{nN}\right)}.
\end{align*}
Thus
$$\lim_{n \rightarrow \infty} A_n = \prod_{k=0}^{\infty} \left( 1 +q^{Nk+a(r)} \right) \frac{g(1)}{\prod_{n=1}^{\infty}\left(1-q^{nN}\right)}.$$
Therefore, by Appell's lemma again,
\begin{equation}
\label{limitF}
\begin{aligned}
\lim_{x \rightarrow 1^-} (1-x) F(x) &= \lim_{n \rightarrow \infty} A_n
\\&= \prod_{k=0}^{\infty} \left( 1 +q^{Nk+a(r)} \right) \frac{g(1)}{\prod_{n=1}^{\infty}\left(1-q^{nN}\right)}.
\end{aligned}
\end{equation}
Finally,
\begin{align*}
f(1) &= \lim_{x \rightarrow 1^-} f(x)
\\&= \lim_{x \rightarrow 1^-} \prod_{n=0}^{\infty} \frac{1-xq^{Nn}}{1-dxq^{Nn+a(r)}} F(x)
\\&= \frac{\prod_{n=1}^{\infty} \left(1-q^{Nn}\right)}{ \prod_{n=0}^{\infty} \left(1-dq^{Nn+a(r)}\right)} \prod_{k=0}^{\infty} \left( 1 +q^{Nk+a(r)} \right) \frac{g(1)}{\prod_{n=1}^{\infty}\left(1-q^{nN}\right)} \ \text{by~\eqref{limitF}}
\\&= \frac{\left(-q^{a(r)};q^N\right)_{\infty}}{\left(dq^{a(r)};q^N\right)_{\infty}} g(1).
\end{align*}
Then by~\eqref{g1},
$$f(1)= \prod_{k=1}^r \frac{(-q^{a(k)};q^N)_{\infty}}{(dq^{a(k)};q^N)_{\infty}}.$$
This completes the proof.
\end{proof}

Now Theorem~\ref{dousse} is a simple corollary of Theorem~\ref{main}.

\begin{proof}[Proof of Theorem~\ref{dousse}]
By Lemma~\ref{conj}, $f_a(1)$ satisfies~\eqref{qdiff}. Therefore
$$ f_{a(1)}(1) = \prod_{k=1}^r \frac{(-q^{a(k)};q^N)_{\infty}}{(dq^{a(k)};q^N)_{\infty}}.$$
But $f_{a(1)}(1)$ is the generating function for overpartitions counted by $E(A'_N;n,k)$, and $$\prod_{k=1}^r \frac{(-q^{a(k)};q^N)_{\infty}}{(dq^{a(k)};q^N)_{\infty}}$$ is the generating function for overpartitions counted by $D(A_N;n,k)$.
Thus $D(A_N;n,k)= E(A'_N;n,k)$ and the theorem is proved.
\end{proof}

\section{Conclusion}
We generalised Andrews' theorem to overpartitions by using recurrences and $q$-difference equations. In~\cite{Generalisation2}, Andrews proved another generalisation of Schur's theorem similar to Theorem~\ref{andrews}. It is likely that similar methods would also work to generalise this theorem to overpartitions. In~\cite{Corteel}, Corteel and Lovejoy proved an even more general theorem of which both of Andrews' theorems are particular cases. It would be interesting to generalise it to overpartitions too, but new techniques might be necessary.

\section*{Acknowledgements}
The author thanks Jeremy Lovejoy for extremely carefully reading the preliminary version of this paper and giving her helpful advice to improve it.

\bibliographystyle{spmpsci}      


\end{document}